%% file: main.tex
\documentclass[11pt, a4paper]{article}
\usepackage{typearea}
\typearea{12}

\usepackage{amsmath,amssymb,amsthm,url}
\usepackage{xcolor}
\usepackage{graphicx}
\usepackage{pgfplots}
\usepackage{subcaption}
\pgfplotsset{
compat=newest, 
cycle list name=exotic 
}

\newtheorem{theorem}{Theorem}[section]
\newtheorem{lemma}[theorem]{Lemma}
\newtheorem{corollary}[theorem]{Corollary}

\newtheorem{proposition}[theorem]{Proposition}

\newcommand{\C}{\mathbb{C}}
\newcommand{\R}{\mathbb{R}}

\newcommand{\bfone}{\mathbf{1}}

\DeclareMathOperator{\capa}{cap}
\DeclareMathOperator{\argmin}{argmin}
\DeclareMathOperator{\tr}{tr}
\DeclareMathOperator{\CH}{CH}

\DeclareMathOperator{\Diag}{Diag}

\usepackage{mathtools}

\DeclarePairedDelimiter{\norm}{\lVert}{\rVert}

\begin{document}
\title{Information Geometry of Operator Scaling}
\author{Takeru~Matsuda\thanks{Statistical Mathematics Unit, RIKEN Center for Brain Science, e-mail: 
\texttt{takeru.matsuda@riken.jp}}
~and~Tasuku~Soma\thanks{Graduate School of Information Science and Technology, the University of Tokyo, e-mail: \texttt{tasuku\_soma@mist.i.u-tokyo.ac.jp}}
}
\date{}

\maketitle

\begin{abstract}
Matrix scaling is a classical problem with a wide range of applications.
It is known that the Sinkhorn algorithm for matrix scaling is interpreted as alternating e-projections from the viewpoint of classical information geometry.
Recently, a generalization of matrix scaling to completely positive maps called operator scaling has been found to appear in various fields of mathematics and computer science, and the Sinkhorn algorithm has been extended to operator scaling.
In this study, the operator Sinkhorn algorithm is studied from the viewpoint of quantum information geometry through the Choi representation of completely positive maps.
The operator Sinkhorn algorithm is shown to coincide with alternating e-projections with respect to the symmetric logarithmic derivative metric, which is a Riemannian metric on the space of quantum states relevant to quantum estimation theory. 
Other types of alternating e-projections algorithms are also provided by using different information geometric structures on the positive definite cone.
\end{abstract}

\input{intro}
\input{matscaling}
\input{opscaling}
\input{generalmarginal}
\input{discussion}
\input{concl}
\input{ack}

\clearpage

\bibliographystyle{plain}
\bibliography{main.bib}

\end{document}

%% file: intro.tex
\section{Introduction}
Given a nonnegative matrix $A \in \R_+^{m \times n}$, the \emph{matrix scaling} problem is to find nonnegative diagonal matrices $L \in \R_+^{m \times m}$ and $R \in \R_+^{n\times n}$ such that 
\[
    (LAR) \bfone_n = \frac{1}{m}\bfone_m \quad \text{and} \quad (LAR)^{\top} \bfone_m = \frac{1}{n}\bfone_n,
\]
where $\bfone_n=(1,\dots,1)^\top$ is the $n$-dimensional all-one vector.
This problem arises in various applications such as Markov chain estimation~\cite{Sinkhorn1964}, data ranking~\cite{Langville2012}, data assimilation~\cite{Reich2019}, and optimal transport~\cite{Peyre2019}.
See \cite{Idel2016} for many other applications.
Sinkhorn~\cite{Sinkhorn1964} proposed an alternating algorithm for matrix scaling. 
Starting from $A^{(0)}=A$, the Sinkhorn algorithm\footnote{The Sinkhorn algorithm is also known as RAS method or iterative proportional fitting procedure (IPFP) \cite{Idel2016}.} iterates {row normalization}
\begin{align*}
    A^{(2k+1)} &= \frac{1}{m}\Diag(A^{(2k)}\bfone_n)^{-1}A^{(2k)}, 
\end{align*}
and {column normalization}
\begin{align*}
    A^{(2k+2)} &= \frac{1}{n} A^{(2k+1)}\Diag((A^{(2k+1})^{\top}\bfone_m)^{-1}.
\end{align*}
Note that $A^{(2k+1)} \bfone_n = m^{-1} \bfone_m$ and $(A^{(2k+2)})^{\top} \bfone = n^{-1}\bfone_n$.
There have been many approaches to prove the convergence of the Sinkhorn algorithm \cite{Idel2016}, such as potential optimization, convex duality, nonlinear Perron--Frobenius theory, and entropy optimization.

Csisz{\'a}r~\cite{Csiszar1975} analyzed the Sinkhorn algorithm by focusing on the geometry of probability distributions.
Specifically, let
\begin{align*}
	\Pi_1 = \left\{ A \in \R_{++}^{m \times n} \mid A \bfone_n = m^{-1}\bfone_m \right\}, \\
	\Pi_2 = \left\{ A \in \R_{++}^{m \times n} \mid A^{\top} \bfone_m = n^{-1}\bfone_n \right\}
\end{align*}
be the spaces of row-normalized and column-normalized positive matrices, respectively.
Then, Csisz{\'a}r~\cite{Csiszar1975} proved that each iteration of the Sinkhorn algorithm coincides with the projection with respect to the Kullback--Leibler divergence $D_{\mathrm{KL}}$:
\begin{align*}
		D_{\mathrm{KL}}(A^{(2k+1)} \mid\mid A^{(2k)}) &= \min_{B \in \Pi_1} D_{\mathrm{KL}}(B \mid\mid A^{(2k)}), \\
		D_{\mathrm{KL}}(A^{(2k+2)} \mid\mid A^{(2k+1)}) &= \min_{B \in \Pi_2} D_{\mathrm{KL}}(B \mid\mid A^{(2k+1)}).
\end{align*}
This is one of the early results of \emph{information geometry} \cite{AmariNagaoka2007,Amari2016}, an interdisciplinary field that provides geometric insights into statistics, information theory, and optimization.
In information geometry, the space of probability distributions is viewed as a Riemannian manifold with the Riemmanian metric given by the Fisher information matrix. 
Then, a pair of dual affine connections called e-connection and m-connection is introduced to this manifold.
In this context, the Sinkhorn algorithm is interpreted as alternating \emph{e-projections} onto $\Pi_1$ and $\Pi_2$.
Further details will be presented in Section~\ref{sec:matscaling}.

Recently, a quantum (non-commutative) generalization of {matrix scaling} called \emph{operator scaling} has been attracting much interests~\cite{Gurvits2004,Garg2018,Garg2019,Georgiou2015}.
Let $\Phi: \C^{n\times n} \to \C^{m\times m}$ be a completely positive linear map given by
\[ 
    \Phi(X) = \sum_{i=1}^k A_i X A_i^\dagger,
\] 
where $A_1,\dots,A_k \in \C^{m \times n}$ and $\dagger$ denotes the Hermitian conjugate.
Its {dual} map $\Phi^*: \C^{m\times m} \to \C^{n\times n}$ is defined as
\[
    \Phi^*(X) = \sum_{i=1}^k A_i^\dagger X A_i.
\]
The {operator scaling} problem is to find non-singular Hermitian matrices $L \in \C^{m \times m}$ and $R \in\C^{n\times n}$ such that
\[
    \Phi_{L,R}(I_n) = \frac{1}{m} I_m \quad \text{and} \quad \Phi_{L,R}^*(I_m) = \frac{1}{n} I_n,
\]
where $I_n$ is the $n$-dimensional identity matrix and
\[
    \Phi_{L,R}(X) = L \Phi(R^{\dagger} X R) L^{\dagger}
\]
is the scaled operator.
The operator scaling problem was originally introduced by Gurvits~\cite{Gurvits2004} to study a major open problem in computational complexity theory called the {Edmonds problem}, which is to determine whether a given matrix subspace contains a non-singular matrix or not. 
Later, operator scaling has been found to appear in surprisingly many fields of mathematics and computer science, such as matrix rank in non-commutative variables~\cite{Garg2019,Ivanyos2017,Ivanyos2018}, Brascamp--Lieb inequalities~\cite{Garg2018}, quantum Schr\"odinger bridge~\cite{Georgiou2015}, multivariate scatter estimation~\cite{Franks2020}, and computational invariant theory~\cite{Allen-Zhu2018}.

Gurvits~\cite{Gurvits2004} extended the {Sinkhorn algorithm} to operator scaling.
Specifically, starting from $\Phi^{(0)} = \Phi$, the operator Sinkhorn algorithm iterates {left normalization}
\begin{align*}
    \Phi^{(2k+1)} &= \Phi_{L,I_n}^{(2k)} \quad \text{where $L=\frac{1}{\sqrt m}\Phi^{(2k)}(I_n)^{-1/2}$,}
\end{align*}
and {right normalization}
\begin{align*}
    \Phi^{(2k+2)} &= \Phi_{I_m,R}^{(2k+1)} \quad \text{where $R=\frac{1}{\sqrt n}(\Phi^{(2k+1)})^*(I_m)^{-1/2}$.}
\end{align*}
Note that $\Phi^{(2k+1)}(I_n)=m^{-1}I_m$ and $(\Phi^{(2k+2)})^*(I_m)=n^{-1}I_n$.
Gurvits~\cite{Gurvits2004} studied the convergence of the operator Sinkhorn algorithm for the case of $m=n$ by extending the {capacity}-based analysis of matrix scaling by \cite{Linial1998}.
Later, full convergence analysis of the operator Sinkhorn algorithm was given by \cite{Garg2019}. 
However, it is still unclear whether the operator Sinkhorn algorithm can be viewed as alternating projections with respect to some divergence measure as in matrix scaling. 
In fact, this question was already posed in the original paper of Gurvits~\cite{Gurvits2004} and even mentioned as ``a major open question'' of operator scaling in the survey of Idel~\cite{Idel2016}.

\subsection{Our contribution}
In this study, we investigate the operator Sinkhorn algorithm from the viewpoint of \emph{quantum information geometry}~\cite{AmariNagaoka2007} by using the Choi representation \cite{Choi1975} of completely positive maps.
Quantum information geometry was originally introduced to study the geometry of the space of quantum states, and it involves many non-trivial problems compared to classical information geometry such as the non-uniqueness of monotone Riemannian metrics and torsion of e-connections.
Among possible Riemannian metrics, the \emph{symmetric logarithmic derivative (SLD) metric} has been found to be relevant to quantum estimation theory such as the quantum Cram\'{e}r--Rao inequality \cite{Helstrom1967}. 
Our main finding is that the operator Sinkhorn algorithm is interpreted as alternating e-projections with respect to the SLD metric. 
Thus, our result is viewed as a generalization of the result by Csisz{\'a}r~\cite{Csiszar1975} to operator scaling. 
Furthermore, we also show that this geometric interpretation of the operator Sinkhorn algorithm even holds for a more general setting, namely, \emph{operator scaling with general marginals}~\cite{Georgiou2015,Franks2018}.

It is still open whether the operator Sinkhorn algorithm can be written as alternating minimization of some divergence.
We discuss alternating minimization algorithms using several divergences known in quantum information theory from the geometric viewpoint and report that these do not coincide with the trajectory of the operator Sinkhorn algorithm.

We believe that our new geometric insight shed new light on operator scaling and may lead to deeper understanding and effective algorithms.



\subsection{Organization}
This paper is organized as follows.
In Section~\ref{sec:matscaling}, we briefly review classical information geometry and explain the geometric interpretation of the matrix Sinkhorn algorithm.
In Section~\ref{sec:opscaling}, we briefly introduce quantum information geometry and present our main finding on the operator Sinkhorn algorithm.
Section~\ref{sec:general-marginals} extends the main result in the previous section to the general marginals setting.
In Section~\ref{sec:discussion}, we discuss other alternating e-projection algorithms using different Riemannian metrics.
In Section~\ref{sec:others}, we touch on another geometric approach to operator scaling in previous work and discuss divergence characterization of the operator Sinkhorn algorithm. 
Finally, we conclude the paper in Section~\ref{sec:conclusion}.

%% file: matscaling.tex
\section{Information geometry of matrix scaling}\label{sec:matscaling}
In this section, we review the result that the Sinkhorn algorithm for matrix scaling is interpreted as alternating e-projections from the viewpoint of classical information geometry, which was originally shown by Csisz{\'a}r~\cite{Csiszar1975}.

\subsection{Classical information geometry}
Classical information geometry provides geometric insights into statistics, information theory, and optimization by regarding the space of probability distributions as a Riemannian manifold with dual affine connections. 
Here, we briefly introduce concepts of classical information geometry.
See \cite{AmariNagaoka2007,Amari2016} for more details.

Let 
\[
    S_{n-1}=\left\{ p=(p_1,\dots,p_n) \mid p_k > 0, \ \sum_{k=1}^n p_k = 1 \right\} \subset \mathbb{R}^n
\]
be the $(n-1)$-dimensional probability simplex.
Each point of $S_{n-1}$ corresponds to a multinomial distribution on $\{ 1,\dots,n \}$. 
We explain the information geometric structure of $S_{n-1}$ following \cite{Fujiwara2015}.

First, we define a Riemannian metric on $S_{n-1}$.
From statistical perspective, the Riemannian metric $g$ on $S_{n-1}$ should be monotone under Markov embeddings.
Chentsov's theorem states that such a Riemannian metric $g$ is uniquely given by the Fisher information matrix (up to constant).
Thus, we adopt the Fisher information matrix as the Riemannian metric tensor on $S_{n-1}$ and call it the \emph{Fisher metric}.
To do calculation with the Fisher metric, the \emph{e-representation} and \emph{m-representation} of tangent vectors are useful.
Recall that each tangent vector of a manifold is identified with a directional derivative operator.
By using this correspondence, the e-representation $X^{(e)} \in \mathbb{R}^n$ and m-representation $X^{(m)} \in \mathbb{R}^n$ of a tangent vector $X$ at $p \in S_{n-1}$ are defined as
\begin{align*}
    X^{(e)} &= (X (\log p_1),\dots,X (\log p_n)), \\
    X^{(m)} &= (X p_1,\dots,X p_n),
\end{align*}
respectively.
Then, the inner product of two tangent vectors $X$ and $Y$ at $p \in S_{n-1}$ with respect to the Fisher metric is calculated as the product-sum of their e-representation and m-representation:
\begin{align}
    g(X, Y) = \sum_{k=1}^n X^{(e)}_k Y^{(m)}_k = \sum_{k=1}^n \frac{1}{p_k} X^{(m)}_k Y^{(m)}_k. \label{Fisher_em}
\end{align}


Next, we introduce a pair of dual affine connections called the \emph{e-connection} and \emph{m-connection} on $S_{n-1}$.
Note that these connections are different from the Levi--Civita connection, which is the unique torsion-free affine connection that preserves the Riemannian metric.
The e-connection and m-connection on $S_{n-1}$ are defined such that their connection coefficients vanish under the coordinate systems $\theta=(\log p_1-\log p_n,\dots,\log p_{n-1}-\log p_n)$ and $\eta=(p_1,\dots,p_{n-1})$, respectively.
In other words, $\theta$ and $\eta$ are e-affine and m-affine coordinate systems, respectively.
Due to the existence of affine coordinate systems, $S_{n-1}$ is both e-flat and m-flat.
In this sense, $S_{n-1}$ is said to be a \emph{dually flat} space.
Since the m-connection on $S_{n-1}$ coincides with the affine connection induced by the natural embedding of $S_{n-1}$ into $\mathbb{R}^n$, the m-geodesic from $p^{(1)} \in S_{n-1}$ to $p^{(2)} \in S_{n-1}$ is given by $\eta(t)=(1-t)\eta^{(1)}+t\eta^{(2)}$, which means that $p_k(t) = (1-t) p_k^{(1)} + t p_k^{(2)}$ for $k=1,\dots,n$.
Similarly, the e-geodesic from $p^{(1)} \in S_{n-1}$ to $p^{(2)} \in S_{n-1}$ is given by $\theta(t) = (1-t) \theta^{(1)} + t \theta^{(2)}$, which means that $\log p_k(t) = (1-t) \log p_k^{(1)} + t \log p_k^{(2)}+C(t)$ for $k=1,\dots,n$, where $C(t)$ is the normalization constant.

Finally, we explain the generalized Pythagorean theorem and e-projection on $S_{n-1}$.
Let
\[
    D_{\mathrm{KL}}(p \mid\mid q) = \sum_{k=1}^n p_k \log \frac{p_k}{q_k}
\]
be the Kullback--Leibler divergence between $p \in S_{n-1}$ and $q \in S_{n-1}$, which satisfies $D_{\mathrm{KL}}(p \mid\mid q) \geq 0$ and $D_{\mathrm{KL}}(p \mid\mid q)=0$ if and only if $p=q$.
The Pythagoream theorem is generalized to the Kullback--Leibler divergence as follows.

\begin{lemma}
For points $p,q,r \in S_{n-1}$, let $\gamma_1$ be the e-geodesic from $p$ to $q$ and $\gamma_2$ be the m-geodesic from $q$ to $r$.
If $\gamma_1$ and $\gamma_2$ are orthogonal with respect to the Fisher metric at $q$, then $D_{\mathrm{KL}}(r \mid\mid p) = D_{\mathrm{KL}}(r \mid\mid q) + D_{\mathrm{KL}}(q \mid\mid p)$.
\end{lemma}

By using this, we obtain the following.

\begin{lemma}\label{lem:eproj}
Let $M$ be an {m-totally geodesic} connected submanifold of $S_{n-1}$ and $p \in S_{n-1}$.
Then, a point $q \in M$ satisfies $D_{\mathrm{KL}}(q \mid\mid p) = \min_{r \in M} D_{\mathrm{KL}}(r \mid\mid p)$ if and only if the e-geodesic from $p$ to $q$ is orthogonal to $M$ at $q$ with respect to the Fisher metric.
\end{lemma}

The unique point $q \in M$ in Lemma~\ref{lem:eproj} is called the \emph{e-projection} of $p$ onto $M$.

\subsection{Sinkhorn as alternating e-projections}
Now, we show that the Sinkhorn algorithm coincides with alternating e-projections.
Note that this result was originally derived by Csisz{\'a}r~\cite{Csiszar1975} \footnote{Csisz{\'a}r used the term ``I-projection" instead of e-projection.}.
Here, we present a proof based on e-geodesics for later reference.

Recall that, given a nonnegative matrix $A \in \R_+^{m \times n}$, the matrix scaling problem is to find nonnegative diagonal matrices $L \in \R_+^{m\times m}$ and $R \in \R_+^{n \times n}$ such that $(LAR) \bfone_n = m^{-1} \bfone_m$ and $(LAR)^{\top} \bfone_m = n^{-1} \bfone_n$. Starting from $A^{(0)}=A$, each iteration of the Sinkhorn algorithm is defined as
\begin{align*}
    A^{(2k+1)} &= \frac{1}{m} \Diag(A^{(2k)}\bfone_n)^{-1}A^{(2k)}, \\
    A^{(2k+2)} &= \frac{1}{n} A^{(2k+1)}\Diag((A^{(2k+1})^{\top}\bfone_m)^{-1}.
\end{align*}

Let
\begin{align*}
    \Pi &= \{ A \in \R_{++}^{m \times n} \mid \bfone_m^{\top} A \bfone_n = 1 \}, \\
	\Pi_1 &= \{ A \in \R_{++}^{m \times n} \mid A \bfone_n = m^{-1} \bfone_m \} \subset \Pi, \\ 
	\Pi_2 &= \{ A \in \R_{++}^{m \times n} \mid A^{\top} \bfone_m = n^{-1} \bfone_n \} \subset \Pi.
\end{align*}
We identify $\Pi$ with $S_{mn-1}$ through vectorization and introduce the corresponding information geometric structure.
Since $\Pi_1$ and $\Pi_2$ are affine subspaces of $\Pi$ under the $\eta$-coordinate, they are m-totally geodesic connected submanifolds of $\Pi$.
From the definition of the Sinkhorn algorithm, it is clear that $A^{(2k+1)} \in \Pi_1$ and $A^{(2k+2)} \in \Pi_2$ for every $k$ if $A^{(0)}=A$ is a positive matrix.
Then, the Sinkhorn algorithm is interpreted as alternating e-projections onto $\Pi_1$ and $\Pi_2$ as follows.

\begin{proposition}\label{prop:matrix-Sinkhorn}
Assume that $A^{(0)}=A$ is a positive matrix.
Then, each iteration of the Sinkhorn algorithm coincides with the e-projection onto $\Pi_1$ or $\Pi_2$.
Namely, the e-geodesic from $A^{(2k)}$ to $A^{(2k+1)}$ (resp. from $A^{(2k+1)}$ to $A^{(2k+2)}$) is orthogonal to $\Pi_1$ (resp. $\Pi_2$) with respect to the Fisher metric for every $k$.
\end{proposition}
\begin{proof}
	The e-geodesic from $A^{(2k)}$ to $A^{(2k+1)}$ is given by
	\begin{align*}
		\log A(t) &= (1-t) \log A^{(2k)} + t \log A^{(2k+1)} + C(t) \bfone_m \bfone_n^{\top},
	\end{align*}
	where $\log$ is applied element-wise\footnote{Thus, it is different from the matrix logarithm.} and $C(t)$ is the normalization constant.
	Therefore, the e-representation of the tangent vector $X$ of this e-geodesic at $A^{(2k+1)}$ is 
	\[
	X^{(e)} = \left. \frac{{\rm d}}{{\rm d} t} \log A(t) \right|_{t=1} = \log A^{(2k+1)} - \log A^{(2k)} + C'(1)  \bfone_m \bfone_n^{\top}.
	\]
	From the definition of the Sinkhorn algorithm,
	\[
		\log A^{(2k+1)}_{ij} - \log A^{(2k)}_{ij} = -\log \left( \sum_{j'} A^{(2k)}_{ij'} \right),
	\]
	which depends only on $i$.
	Hence, each row of $X^{(e)}$ is parallel to $\bfone_n$.
	On the other hand, from the definition of $\Pi_1$, each row of the m-representation $Y^{(m)}$ of a tangent vector $Y$ of $\Pi_1$ is orthogonal to $\bfone_n$.
	Therefore, from \eqref{Fisher_em}, $X$ and $Y$ are orthogonal to each other with respect to the Fisher metric.
    Hence, the e-geodesic from $A^{(2k)}$ to $A^{(2k+1)}$ is orthogonal to $\Pi_1$ with respect to the Fisher metric.
    The proof for the e-geodesic from $A^{(2k+1)}$ to $A^{(2k+2)}$ is similar.
\end{proof}

For $A,B \in \Pi$, let
\[
    D_{\mathrm{KL}}(A \mid\mid B) = \sum_{i,j} A_{ij} \log \frac{A_{ij}}{B_{ij}}
\]
be the Kullback--Leibler divergence.
From Proposition~\ref{prop:matrix-Sinkhorn} and Lemma~\ref{lem:eproj}, we obtain the following.

\begin{corollary}\label{cor:KL-Sinkhorn}
    Assume that $A^{(0)}=A$ is a positive matrix.
	Then, each iteration of the Sinkhorn algorithm provides the unique minimizer of the Kullback--Leibler divergence:
	\begin{align*}
		D_{\mathrm{KL}}(A^{(2k+1)} \mid\mid A^{(2k)}) &= \min_{B \in \Pi_1} D_{\mathrm{KL}}(B \mid\mid A^{(2k)}), \\ D_{\mathrm{KL}}(A^{(2k+2)} \mid\mid A^{(2k+1)}) &= \min_{B \in \Pi_2} D_{\mathrm{KL}}(B \mid\mid A^{(2k+1)}).
	\end{align*}
\end{corollary}

%% file: opscaling.tex
\section{Information geometry of operator scaling}\label{sec:opscaling}
In this section, we present our main result on a quantum information geometric interpretation of the operator Sinkhorn algorithm.

\subsection{Prerequisites from matrix analysis}
Here, we introduce necessary concepts of matrix analysis.
See \cite{HiaiPetz2014,Bhatia2009positive} for more details.


For two matrices $A = (a_{ij}) \in \C^{p\times q}$ and $B=(b_{kl}) \in \C^{r \times s}$, their \emph{Kronecker product} $A \otimes B \in \C^{pr \times qs}$ is the partitioned matrix given by 
\begin{align*}
    A \otimes B = 
    \begin{bmatrix}
        a_{11} B & \cdots & a_{1q} B \\
        \vdots & \ddots & \vdots \\
        a_{p1} B & \cdots & a_{pq}B 
    \end{bmatrix}. 
\end{align*}

For a partitioned matrix
\begin{align*}
    \begin{bmatrix}
        A_{11} & \cdots & A_{1n} \\
        \vdots & \ddots & \vdots \\
        A_{n1} & \cdots & A_{nn} 
    \end{bmatrix}
\end{align*}
with $A_{ij} \in \C^{m \times m}$ for every $(i,j)$, its \emph{partial traces} are defined as 
\begin{align*}
    \tr_1
    \begin{bmatrix}
        A_{11} & \cdots & A_{1n} \\
        \vdots & \ddots & \vdots \\
        A_{n1} & \cdots & A_{nn} 
    \end{bmatrix} &= \sum_{i=1}^n A_{ii}  \in \C^{m \times m}, \\
    \tr_2
    \begin{bmatrix}
        A_{11} & \cdots & A_{1n} \\
        \vdots & \ddots & \vdots \\
        A_{n1} & \cdots & A_{nn} 
    \end{bmatrix} &= 
    \begin{bmatrix}
        \tr A_{11} & \cdots & \tr A_{1n} \\
        \vdots & \ddots & \vdots \\
        \tr A_{n1} & \cdots & \tr A_{nn} 
    \end{bmatrix} \in \C^{n \times n}.
\end{align*}

A linear map $\Phi: \C^{n \times n} \to \C^{m \times m}$ is said to be \emph{completely positive} if it has the Kraus representation:
\[ 
    \Phi(X) = \sum_{i=1}^k A_i X A_i^\dagger,
\] 
where $A_1,\dots,A_k \in \C^{m \times n}$ and $\dagger$ denotes the Hermitian conjugate.
Then, its {dual} map $\Phi^*: \C^{m\times m} \to \C^{n\times n}$ is also completely positive with the Kraus representation
\[
    \Phi^*(X) = \sum_{i=1}^k A_i^\dagger X A_i.
\]
In quantum information theory, it is known that any quantum operation is described by a trace-preserving completely positive (TPCP) map \cite{Holevo2013,Petz2008}.

For a linear map $\Phi: \C^{n \times n} \to \C^{m \times m}$, its \emph{Choi representation} $\CH(\Phi) \in \C^{mn \times mn}$ is defined as
\begin{align*}
    \CH(\Phi) = \sum_{i,j=1}^n E_{ij} \otimes \Phi(E_{ij}) =     \begin{bmatrix}
        \Phi(E_{11}) & \cdots & \Phi(E_{1n}) \\
        \vdots & \ddots & \vdots \\
        \Phi(E_{n1}) & \cdots & \Phi(E_{nn}) 
    \end{bmatrix},
\end{align*}
where $E_{ij}$ is the matrix unit with 1 in the $(i,j)$-th entry and 0s elsewhere.
From definition, $\tr_1\CH(\Phi)=\Phi(I_n)$ and $\tr_2\CH(\Phi)=\Phi^*(I_m)$.
In addition, Choi~\cite{Choi1975} showed the following important property.

\begin{lemma}[Choi~\cite{Choi1975}]\label{lem:choi}
$\Phi$ is completely positive if and only if $\CH(\Phi)$ is positive semidefinite.
\end{lemma}

For an Hermitian matrix $A \in \C^{n\times n}$ (i.e., $A = A^\dagger$), let 
\begin{align*}
    A = P \Lambda P^\dagger, \quad \Lambda = \mathrm{diag} (\lambda_1,\dots,\lambda_n)
\end{align*}
be its spectral decomposition.
Assume that the eigenvalues $\lambda_1,\dots,\lambda_n$ of $A$ are contained in the interval $[a,b]$.
Then, for a function $f: [a,b] \to \R$, we define the Hermitian matrix $f(A) \in \C^{n\times n}$ by
\begin{align*}
    f(A) = P f(\Lambda) P^\dagger, \quad f(\Lambda)=\mathrm{diag}(f(\lambda_1),\dots,f(\lambda_n)).
\end{align*}
In particular, for a positive semidefinite matrix $A$ and the square root function $f(x) = x^{1/2}$, we denote $f(A)$ by $A^{1/2}$.
A function $f$ is said to be \emph{operator monotone} if $f(A) \succeq f(B)$ holds for every $A$ and $B$ satisfying $A \succeq B$.

Given $A \in \C^{n \times n}$ and $Q \in \C^{n \times n}$, the (continuous) \emph{Lyapunov equation} $A^\dagger X + X A = Q$ has a unique solution of $X$ given by
\begin{align}\label{eq:integral}
    X = \int_{0}^{\infty} e^{-tA^\dagger} Q e^{tA} dt,
\end{align}
if $A$ is positively stable (i.e., all the eigenvalues of $A$ are contained in the open right half-plane of $\C$).
Furthermore, if both $A$ and $Q$ are Hermitian, then the solution $X$ is also Hermitian.

For positive definite matrices $A \in \C^{n \times n}$ and $B \in \C^{n \times n}$, their \emph{geometric mean} $A\# B \in \C^{n \times n}$ is defined as 
\begin{align*}
    A \# B = A^{1/2}(A^{-1/2}B A^{-1/2})^{1/2} A^{1/2}.
\end{align*}
It is known that $A \# B$ is positive definite and $A\# B = B \# A$.
Furthermore, $A\# B$ is the unique positive definite solution of the \emph{Riccati equation} $X A^{-1} X = B$.

\subsection{Quantum information geometry}
The theory of classical information geometry has been extended to quantum systems.
Here, we briefly introduce several concepts of quantum information geometry.
See \cite{AmariNagaoka2007} for more details.

In quantum information theory~\cite{Holevo2013,Petz2008}, each quantum system is associated with a complex Hilbert space.
For example, a qubit system is associated with $\mathbb{C}^2$.
A state of a quantum system associated with $\C^n$ is described by a \emph{density matrix} $\rho \in \C^{n \times n}$, which is a positive semidefinite Hermitian matrix of trace one.
Let 
\begin{align*}
    S(\C^n)=\left\{ \rho \in \C^{n \times n} \mid \rho \succ 0, \ \tr \rho = 1 \right\}
\end{align*}
be the set of positive definite density matrices on $\C^n$.
In quantum information geometry, the set $S(\C^n)$ is regarded as a $(n^2-1)$-dimensional Riemannian manifold.
We explain the information geometric structure of $S(\C^n)$ following \cite{Fujiwara2015}.

First, we define a Riemannian metric on $S(\C^n)$.
Whereas the Riemannian metric is uniquely specified as the Fisher metric by Chentsov's theorem in classical information geometry, such uniqueness no longer holds in quantum information geometry.
Specifically, Petz~\cite{Petz1996} investigated Riemannian metrics on $S(\C^n)$ that are monotone under TPCP maps and showed that there is a one-to-one correspondence between monotone Riemannian metrics and operator monotone functions. 
Among them, here we focus on the \emph{symmetric logarithmic derivative} (SLD) metric, which corresponds to the function $t \mapsto (t+1)/2$.
For a tangent vector $X$ at $\rho \in S(\C^n)$, its e-representation $X^{(e)} \in \C^{n \times n}$ is defined by the unique Hermitian solution of the Lyapunov equation
\begin{align}\label{eq:SLD}
    X^{(e)} \rho + \rho X^{(e)} = 2 X \rho.
\end{align}
From \eqref{eq:integral}, $X^{(e)}$ is given by
\[
    X^{(e)} = 2 \int_{0}^{\infty} e^{-t\rho} (X \rho) e^{t\rho} {\rm d} t.
\]
On the other hand, the m-representation $X^{(m)} \in \C^{n \times n}$ of a tangent vector $X$ at $\rho \in S(\C^n)$ is defined by $X \rho$.
Then, the inner product of two tangent vectors $X$ and $Y$ at $\rho \in S(\C^n)$ with respect to the SLD metric is given by 
\[
g^{\mathrm{S}}(X,Y) = \tr (X^{(e)} Y^{(m)}) = 2 \tr \int_{0}^{\infty} e^{-t\rho} (X \rho) e^{t\rho} (Y \rho) {\rm d} t.
\]
The SLD metric has been found to play a central role in extending the Cram\'{e}r--Rao inequality to quantum estimation \cite{Petz2008}.

Next, we introduce dual affine connections on $S(\C^n)$.
Similarly to the classical case, the m-connection on $S(\C^n)$ is defined as the torsion-free affine connection induced by the natural embedding of $S(\C^n)$ into $\R^{n^2}$:
\[
\rho \mapsto (\rho_{11},\dots,\rho_{nn},{\rm Re} \ \rho_{12},{\rm Im} \ \rho_{12},\dots,{\rm Re} \ \rho_{n-1,n},{\rm Im} \ \rho_{n-1,n}).
\]
Thus, the m-geodesic from $\rho_1$ to $\rho_2$ is simply given by $\rho(t) = (1-t)\rho_1 + t\rho_2$ and the space $S(\C^n)$ is m-flat.
On the other hand, the e-connection on $S(\C^n)$ is not unique and depends on which Riemannian metric is introduced to $S(\C^n)$.
Namely, for each Riemannian metric $g$ on $S(\C^n)$, the e-connection on $S(\C^n)$ is defined to be the dual of the m-connection with respect to $g$.

Let us adopt the SLD metric now.
The following lemma is known in the literature~\cite{AmariNagaoka2007,Fujiwara2015}, but we provide a self-contained proof for the completeness.
\begin{lemma}\label{lem:SLD-geodesic}
The e-geodesic from $\rho_1$ to $\rho_2$ is explicitly given by \begin{align}\label{eq:SLD-geodesic}
    \rho(t) = C(t) K^t \rho_1 K^t,
\end{align} 
where $K = \rho_1^{-1} \# \rho_2$ is the matrix geometric mean of $\rho_1^{-1}$ and $\rho_2$ and $C(t) = 1/\tr (K^t \rho_1 K^t)$ is the normalization constant for $\tr \rho(t)=1$.
\end{lemma}
\begin{proof}
For $\rho, \sigma \in \Pi$ and a tangent vector $M \in T_\rho(S(\C^n))$ in m-representation,
let $\Pi^{(m)}_{\rho, \sigma} M$ be the m-parallel transport of $M$ to $T_\sigma(S(\C^n))$.
Similarly, $\Pi^{(e)}_{\rho, \sigma} L$ be the e-parallel transport of a tangent vector $L \in T_\rho(S(\C^n))$ (in e-representation) to $T_\sigma(S(\C^n))$.
By the definition of m-connection, $\Pi^{(m)}_{\rho, \sigma} M = M$.
We first show a formula for e-parallel transport
\begin{align}\label{eq:e-parallel-transport}
    \Pi^{(e)}_{\rho, \sigma} L = L - \tr(\sigma L)I.
\end{align}
Since m- and e-connections are dual, we must have 
\[
    g^{\mathrm{S}}\left(\Pi^{(m)}_{\rho, \sigma} M, \Pi^{(e)}_{\rho, \sigma} L\right)_\sigma = 
    g^{\mathrm{S}}(M, L)_\rho.
\]
Therefore, we must have
\[
    \tr(M \Pi^{(e)}_{\rho, \sigma} L) = \tr(M L).
\]
Putting $\Pi^{(e)}_{\rho, \sigma} L = L + L'$, we have $\tr(M L') = 0$ for all Hermite $M$ with $\tr M = 0$. 
Thus, $L' = \alpha I$ for some $\alpha \in \C$.
Furthermore, since $\Pi^{(e)}_{\rho, \sigma} L = L + \alpha I \in T_\sigma(S(\C^n))$, we have $\tr(\sigma (L+\alpha I)) = \tr(\sigma L) + \alpha \tr\sigma = \tr(\sigma L) + \alpha = 0$.
Thus, $\alpha = - \tr(\sigma L)$ and we have obtained the desired formula \eqref{eq:e-parallel-transport}.

Now, let us show that \eqref{eq:geodesic} is the e-geodesic from $\rho_1$ to $\rho_2$.
It is easy to see that $\rho(0) = \rho_1$ and $\rho(1) = \rho_2$ by the definition of matrix geometric mean.
To show that $\rho(t)$ is indeed the e-geodesic, it is suffices to show that $\rho(t)$ is e-totally geodesic, i.e., $\Pi^{(e)}_{\rho_1, \rho(t)} \dot{\rho}(0)^{(e)} = \dot{\rho}(t)^{(e)}$, where $\dot{\rho}(t)^{(e)}$ is the e-representation of a tangent vector $\dot{\rho}(t)$.
Since $\rho(t) = C(t)K^t \rho_1 K^t$, we have
\begin{align*}
    \dot{\rho}(t) 
    &= \dot{C}(t)K^t\rho_1 K^t + C(t)(\log K \cdot K^t \rho_1 K^t + K^t \rho_1 K^t \cdot \log K) \\
    &= \frac{\dot{C}(t)}{C(t)} \rho(t) + \log K \rho(t) + \rho(t) \log K \\
    &= \frac{1}{2} \left[ \left(\frac{\dot{C}(t)}{C(t)} I + 2 \log K \right) \rho(t) + \rho(t) \left(\frac{\dot{C}(t)}{C(t)} I + 2 \log K \right) \right].
\end{align*} 
Thus, 
\begin{align*}
    \dot{\rho}(t)^{(e)} &= \frac{\dot{C}(t)}{C(t)} I + 2 \log K = -2 \tr(\rho(t)\log K) I + 2 \log K, \\
    \dot{\rho}(0)^{(e)} &= \dot{C}(0) I + 2 \log K, 
\end{align*}
where we used
\[
    \frac{\dot{C}(t)}{C(t)} 
    = - \frac{\tr(\log K \cdot K^t \rho_1 K^t + K^t \rho_1 K^t \cdot \log K)}{\tr(K^t \rho K^t)} 
    = -2 \tr(\rho(t)\log K) 
\]
by the direct calculation.
Now, using the formula \eqref{eq:e-parallel-transport} for e-parallel transport,
\begin{align*}
    \Pi^{(e)}_{\rho_1, \rho(t)} \dot{\rho}(0)^{(e)} 
    &= \dot{\rho}(0)^{(e)} - \tr(\rho(t) \dot{\rho}(0)^{(e)}) I \\
    &= \dot{C}(0) I + 2 \log K - \tr (\rho(t) \dot{C}(0) + 2\rho(t)\log K) I\\
    &= -2 \tr(\rho(t)\log K) I + 2 \log K \\
    &= \dot{\rho}(t)^{(e)}.
\end{align*}
This completes the proof.
\end{proof}

Note that this e-connection is not torsion-free\footnote{The e-connection on $S(\C^n)$ is torsion-free if and only if the Bogoliubov--Kubo--Mori metric \eqref{eq:Bmetric} is adopted \cite{AmariNagaoka2007,Fujiwara2015}.}.
Thus, the space $S(\C^n)$ is not dually flat under the SLD metric and it is not clear whether there exists some canonical divergence and analogue of the generalized Pythagoream theorem in this case.

\subsection{Main result}
Now, we present our main result on the operator Sinkhorn algorithm.
Recall that, given a completely positive linear map $\Phi: \C^{n\times n} \to \C^{m\times m}$, the {operator scaling} problem is to find nonsingular Hermitian matrices $L\in \C^{m \times m}$ and $R \in \C^{n\times n}$ such that $\Phi_{L,R}(I_n) = m^{-1} I_m$ and $\Phi_{L,R}^*(I_m) = n^{-1} I_n$, where $\Phi_{L,R}(X) = L\Phi(R^{\dagger}XR)L^{\dagger}$.
Starting from $\Phi^{(0)} = \Phi$, each iteration of the operator Sinkhorn algorithm is defined as
\begin{align*}
    \Phi^{(2k+1)} &= \Phi_{L,I_n}^{(2k)} \quad \text{where $L=\frac{1}{\sqrt{m}}\Phi^{(2k)}(I_n)^{-1/2}$,}  \\
    \Phi^{(2k+2)} &= \Phi_{I_m,R}^{(2k+1)} \quad \text{where $R=\frac{1}{\sqrt{n}}(\Phi^{(2k+1)})^*(I_m)^{-1/2}$.}
\end{align*}

For quantum information geometric consideration, we identify a completely positive map with its Choi representation.
Then, we obtain the following formula.

\begin{lemma}\label{lem:choi2}
The Choi representation of $\Phi_{L,R}$ is given by
    \begin{align*}
\CH(\Phi_{L,R})=(R^{\dagger} \otimes L) \CH(\Phi) (R \otimes L^{\dagger}).
    \end{align*}
\end{lemma}
\begin{proof}
Since $\Phi_{L,I_n}(X) = L \Phi(X) L^{\dagger}$,
    \begin{align*}
    \CH(\Phi_{L,I_n}) = \sum_{i,j} E_{ij} \otimes L \Phi(E_{ij}) L^{\dagger} = (I_n \otimes L) \CH(\Phi) (I_n \otimes L^{\dagger}),
    \end{align*}
where we used the formula $(A\otimes B)(C \otimes D) = AC \otimes BD$ in the second equality.
    
On the other hand, since $\Phi_{I,R}(X) = \Phi(R^{\dagger} X R)$,
    \begin{align*}
    \CH(\Phi_{I_m,R}) = \sum_{i,j} E_{ij} \otimes \Phi(R^{\dagger} E_{ij} R) = (R^{\dagger} \otimes I_m) \CH(\Phi) (R \otimes I_m),
    \end{align*}
where we used 
    \begin{align*}
    \Phi(R^{\dagger} E_{ij} R) = \Phi \left( \sum_{k,l} \bar{R}_{ik} R_{jl} E_{kl} \right) = \sum_{k,l} \bar{R}_{ik} R_{jl} \Phi(E_{kl}) = \sum_{k,l} \bar{R}_{ki} R_{lj} \CH(\Phi)_{kl}.
    \end{align*}

Therefore, from $\Phi_{L,R} = (\Phi_{L,I_n})_{I_m,R}$,
    \begin{align*}
    \CH(\Phi_{L,R}) &= (R^{\dagger} \otimes I_m) \CH(\Phi_{L,I_n}) (R \otimes I_m) \\
    &= (R^{\dagger} \otimes I_m) (I_n \otimes L) \CH(\Phi) (I_n \otimes L^{\dagger}) (R \otimes I_m) \\
    &=(R^{\dagger} \otimes L) \CH(\Phi) (R \otimes L^{\dagger}).
    \end{align*}
\end{proof}

Let $\rho_k=\CH(\Phi^{(k)})\in \C^{mn \times mn}$ be the Choi representation of $\Phi^{(k)}$.
Note that $\rho_k \succeq 0$ from Lemma~\ref{lem:choi}.
By using Lemma~\ref{lem:choi2}, the operator Sinkhorn algorithm is rewritten as follows.

\begin{lemma}\label{lem:opsink}
    Each iteration of the operator Sinkhorn algorithm is given by
    \begin{align*}
    \rho_{2k+1} = (I_n \otimes L)\rho_{2k} (I_n \otimes L),&\quad \text{$L=\frac{1}{\sqrt{m}}(\tr_1\rho_{2k})^{-1/2}$,} \\ 
    \rho_{2k+2} = (R \otimes I_m) \rho_{2k+1} (R \otimes I_m),&\quad \text{$R=\frac{1}{\sqrt{n}}(\tr_2\rho_{2k+1})^{-1/2}$}. 
    \end{align*}
\end{lemma}

Let
\begin{align*}
	\Pi &= \{ \rho \in \C^{mn \times mn} \mid \rho \succ O, \ {\rm tr} \ \rho =1 \}, \\
	\Pi_1 &= \{ \rho \in \C^{mn \times mn} \mid \rho \succ O, \ {\rm tr}_1 \rho = m^{-1} I_m \} \subset \Pi, \\
	\Pi_2 &= \{ \rho \in \C^{mn \times mn} \mid \rho \succ O, \ {\rm tr}_2 \rho = n^{-1} I_n \} \subset \Pi.
\end{align*}
We identify $\Pi$ with $S(\C^{mn})$ and introduce the corresponding information geometric structure with the SLD metric.
From the definition of $\Pi_1$ and $\Pi_2$, they are m-totally geodesic connected submanifolds of $\Pi$.

Since the partial trace of a positive definite matrix is also positive definite, we have the following.

\begin{lemma}\label{lem:oppos}
    If $\rho_0 \succ O$, then $\rho_{2k+1} \in \Pi_1$ and $\rho_{2k+2} \in \Pi_2$ for every $k$. 
\end{lemma}

Then, the operator Sinkhorn algorithm is interpreted as alternating e-projections as follows.

\begin{theorem}\label{thm:opscaling}
    Assume that $\rho_0 \succ O$.
    Then, each iteration of the operator Sinkhorn algorithm coincides with the e-projection onto $\Pi_1$ or $\Pi_2$ with respect to the SLD metric.
    Namely, $\rho_{2k+1}$ (resp. $\rho_{2k+2}$) is the unique point in $\Pi_1$ (resp. $\Pi_2$) such that the e-geodesic from $\rho_{2k}$ to $\rho_{2k+1}$ (resp. from $\rho_{2k+1}$ to $\rho_{2k+2}$) is orthogonal to $\Pi_1$ (resp. $\Pi_2$) with respect to the SLD metric for every $k$.
\end{theorem}
\begin{proof}
    From Lemma~\ref{lem:oppos}, we have $\rho_k \succ O$ for every $k$.
    In the following, we prove the statement for $\rho_{2k+1}$.
    The proof for $\rho_{2k+2}$ is similar.
    
    Let $K=\rho_{2k}^{-1}\# \rho_{2k+1}$.
    Since $\rho_{2k+1} = (I_n \otimes {L} )\rho_{2k} (I_n \otimes {L})$ with ${L} = (\tr_1 \rho_{2k})^{-1/2}/\sqrt{m} \succ O$ from Lemma~\ref{lem:opsink},
    \begin{align*}
        K &= \rho_{2k}^{-1/2} (\rho_{2k}^{1/2} \rho_{2k+1} \rho_{2k}^{1/2})^{1/2} \rho_{2k}^{-1/2} \\
        &= \rho_{2k}^{-1/2} (\rho_{2k}^{1/2} (I_n \otimes L )\rho_{2k} (I_n \otimes {L}) \rho_{2k}^{1/2})^{1/2} \rho_{2k}^{-1/2} \\
        &= \rho_{2k}^{-1/2} \cdot \rho_{2k}^{1/2} (I_n \otimes {L} )\rho_{2k}^{1/2} \cdot \rho_{2k}^{-1/2} \\
        &= I_n \otimes {L}.
    \end{align*}
    By Lemma~\ref{lem:SLD-geodesic}, the e-geodesic from $\rho_{2k}$ to $\rho_{2k+1}$ with respect to the SLD metic is given by 
    \[
        \rho(t) = C(t) K^t \rho_{2k} K^t.
    \]
    Thus, 
    \[
        \left. \frac{{\rm d}}{{\rm d} t} \rho(t) \right|_{t=1} = C'(1) \rho_{2k+1} + C(1) (\log K) \rho_{2k+1} + C(1) \rho_{2k+1} (\log K).
    \]
    Therefore, the e-representation $X^{(e)}$ of the tangent vector $X$ of this e-geodesic at $\rho_{2k+1}$ is the solution of the Lyapunov equation 
    \[
        X^{(e)} \rho_{2k+1} + \rho_{2k+1} X^{(e)} = 2 \left( C'(1) \rho_{2k+1} + C(1)  (\log K) \rho_{2k+1} + C(1) \rho_{2k+1} (\log K) \right),
    \]
    which has a unique solution since $\rho_{2k+1} \succ O$.
    Namely, we must have 
    \[
        X^{(e)} = C'(1) I_{mn} + 2 C(1) \log K = I_n \otimes (C'(1) I_m + 2 C(1) \log {L}).
    \]
    On the other hand, from the definition of $\Pi_1$, the m-representation $Y^{(m)}$ of a tangent vector $Y$ of $\Pi_1$ satisfies $\tr_1 Y^{(m)}=O$.
	Therefore, 
	\[
	g^{{\rm S}} (X,Y) = \tr (X^{(e)} Y^{(m)}) = \tr [(C'(1) I_m + 2 C(1) \log {L}) (\tr_1 Y^{(m)})] = 0.
	\]
    Hence, the e-geodesic from $\rho_{2k}$ to $\rho_{2k+1}$ is orthogonal to $\Pi_1$ with respect to the SLD metric.

    Conversely, suppose that the e-geodesic from $\rho_{2k}$ to $\rho \in \Pi_1$ is orthogonal to $\Pi_1$ with respect to the SLD metric.
    Then, by following the above argument reversely, we must have $\rho_{2k}^{-1} \# \rho = I_n \otimes M$ with some $M$ and thus $\rho = (I_n \otimes M) \rho_{2k} (I_n \otimes M)$.
    Since $\rho_{2k}$ is the Choi representation of $\Phi^{(2k)}$,
    \begin{align*}
        \rho &= (I_n \otimes M) \left( \sum_{i,j} E_{ij}\otimes\Phi^{(2k)}(E_{ij}) \right) (I_n \otimes M) \\
        &= \sum_{i,j} E_{ij} \otimes (M\Phi^{(2k)}(E_{ij})M).
    \end{align*}
    Therefore,
    \begin{align*}
        \tr_1\rho &= \sum_{i,j} \tr(E_{ij}) \cdot M\Phi^{(2k)}(E_{ij})M \\
        &= \sum_{i} M\Phi^{(2k)}(E_{ii})M \tag{since $\tr(E_{ij}) = 1$ if $i=j$ and $0$ otherwise}\\
        &= M\left(\sum_{i}\Phi^{(2k)}(E_{ii}) \right)M  \\
        &= M\Phi^{(2k)}(I_n) M \tag{By the linearity of $\Phi^{(2k)}$} \\
        &= M (\tr_1 \rho_{2k}) M. 
    \end{align*}
    Since $\rho \in \Pi_1$, we must have $\tr_1\rho = m^{-1} I_m$.
    Hence, from the uniqueness of the solution of the Riccati equation, we must have $M = (\tr_1 \rho_{2k})^{-1/2}/\sqrt{m}={L}$ and thus $\rho = \rho_{2k+1}$.
\end{proof}

The SLD metric does not induce a dually flat structure, because the e-connection does not become torsion-free.
Traditionally, the theory of information geometry has been mostly developed in the dually flat setting \cite{AmariNagaoka2007,Amari2016} and the theory for statistical manifolds admitting torsion \cite{Matsuzoe2010,Henmi2019} is still largely unexplored \cite{Fujiwara2015}.
Therefore, although we found that the operator Sinkhorn algorithm coincides with alternating e-projections with respect to the SLD metric, it is still unclear whether it can be viewed as alternating minimization of some divergence.
We will discuss it more in details in Section~\ref{sec:discussion}.




%% file: generalmarginal.tex
\section{Operator scaling with general marginals}\label{sec:general-marginals}
In this section, we generalize Theorem~\ref{thm:opscaling} for operator scaling to a more general setting, namely, \emph{operator scaling with general marginals}~\cite{Georgiou2015,Franks2018}.
In this problem, we are given a CP map $\Phi : \C^{n\times n} \to \C^{m \times m}$, and positive semidefinite Hermite matrices $P \in \C^{m \times m}$ and $Q \in \C^{n \times n}$ with unit trace.
The goal is to find nonsingular Hermitian matrices $L \in \C^{m \times m}$ and $R \in \C^{n \times n}$ such that 
\begin{align}
    \Phi_{L, R}(I_n) = P \quad\text{and}\quad \Phi_{L,R}^*(I_m) = Q.
\end{align}
The original operator scaling is the special case that $P = \frac{1}{m} I_m$ and $Q = \frac{1}{n}I_n$.
To distinguish the general marginals setting with the original setting, we call the latter the \emph{doubly stochastic setting} in this section.

Georgiou and Pavon~\cite{Georgiou2015} studied the special case of operator scaling with general marginals that $P = \frac{1}{m} I_m$ and $Q$ is general. 
They proved that a natural generalization of the operator Sinkhorn algorithm converges to a solution (if it exists).
Franks~\cite{Franks2018} first investigated the case of general $P$ and $Q$, but scaling matrices $L$ and $R$ are restricted to upper-triangular matrices for some technical reason:
Such restriction is necessary to define a generalization of capacity for the general marginals setting, which in turn enables him to analyze a generalized operator Sinkhorn algorithm.

Below, we show that our information geometric result can be generalized to operator scaling with general marginals without restriction of scaling matrices $L$ and $R$.

\subsection{Operator Sinkhorn algorithm for general marginals}
The operator Sinkhorn algorithm for general marginals is similar to the doubly stochastic setting. 
Let $\Phi$ be the input CP map and we assume that $\Phi^*(I_m)= Q$. 
We can assume this without loss of generality, because $\Phi^*_{I_n, R}(I_m) = R^\dagger \Phi^*(I_m) R = Q$, where $R = \Phi^*(I_m)^{-1} \# Q$ is the matrix geometric mean.
Starting with $\Phi^{(0)} = \Phi$, we iterate left normalization
\begin{align*}
    \Phi^{(2k+1)} &= \Phi_{L,I_n}^{(2k)} \quad \text{where $L=\Phi^{(2k)}(I_n)^{-1}\# P$,}
\end{align*}
and {right normalization}
\begin{align*}
    \Phi^{(2k+2)} &= \Phi_{I_m,R}^{(2k+1)} \quad \text{where $R=(\Phi^{(2k)})^*(I_m)^{-1}\# Q$.}
\end{align*}
for $k = 0, 1, 2, \dots$.
We note that $\Phi^{(2k)}(I_n) = P$ and $\Phi^{(2k+1)}(I_m)=Q$ for $k = 0, 1, 2, \dots$ by the definition of $L$ and $R$.

\subsection{Main theorem for general marginals}
One can extend Theorem~\ref{thm:opscaling} for the doubly stochastic setting to the general marginals setting.
Let $\rho_k$ be the Choi representation of $\Phi^{(k)}$ and define 
\begin{align*}
	\Pi &= \{ \rho \in \C^{mn \times mn} \mid \rho \succ O, \ {\rm tr} \ \rho =1 \}, \\
	\Pi_1 &= \{ \rho \in \C^{mn \times mn} \mid \rho \succ O, \ {\rm tr}_1 \rho = P \} \subset \Pi, \\
	\Pi_2 &= \{ \rho \in \C^{mn \times mn} \mid \rho \succ O, \ {\rm tr}_2 \rho = Q \} \subset \Pi.
\end{align*}
Then, $\rho_{2k+1} \in \Pi_1$ and $\rho_{2k+2} \in \Pi_2$ for $k = 0, 1, 2, \dots$.
Furthermore, $\Pi_1$ and $\Pi_2$ are m-totally geodesic submanifolds.

\begin{theorem}\label{thm:general-marginals}
    Assume that $\rho_0, P, Q \succ O$.
    Then, each iteration of the operator Sinkhorn algorithm for general marginals coincides with the e-projection onto $\Pi_1$ or $\Pi_2$ with respect to the SLD metric.
    Namely, $\rho_{2k+1}$ (resp. $\rho_{2k+2}$) is the unique point in $\Pi_1$ (resp. $\Pi_2$) such that the e-geodesic from $\rho_{2k}$ to $\rho_{2k+1}$ (resp. from $\rho_{2k+1}$ to $\rho_{2k+2}$) is orthogonal to $\Pi_1$ (resp. $\Pi_2$) with respect to the SLD metric for every $k$.
\end{theorem}
\begin{proof}
The proof is mostly same as Theorem~\ref{thm:opscaling}.
Let us prove the theorem for $\rho_{2k+1}$ since the other is similar.
The main difference is proving that $K = \rho_{2k}^{-1}\# \rho_{2k+1}$ satisfies $K = I_n \otimes L$ for $L = (\tr_1 \rho_{2k})^{-1}\# P$.
By the definition of $\rho_{2k+1}$ and Lemma~\ref{lem:choi2}, we have $\rho_{2k+1}=(I_n\otimes L)\rho_{2k}(I_n\otimes L)$.
Further, $L \succ O$ because both $P$ and $\tr_1 \rho_{2k}$ are positive definite and the matrix geometric mean of two positive definite matrices is again positive definite.
By the uniqueness of the positive definite solution of Riccati equations, we must have $K = I_n \otimes L$.
The rest of the proof is the same as Theorem~\ref{thm:opscaling}.
\end{proof}

%% file: discussion.tex
\section{Other alternating e-projections algorithms}\label{sec:discussion}
We showed that the operator Sinkhorn algorithm is interpreted as alternating e-projections with respect to the symmetric logarithmic derivative metric.
In this section, we provide other alternating e-projections algorithms based on other information geometric structures on the positive definite cone.

\subsection{Alternating e-projections with respect to Bogoliubov--Kubo--Mori metric}\label{subsec:kubo}
In quantum information geometry, the \emph{Bogoliubov--Kubo--Mori metric}~\cite{Petz1993} is defined as
\begin{align}\label{eq:Bmetric}
    g^{\mathrm{B}}(X, Y) = \tr[X(\rho) \log(Y(\rho))],
\end{align}
and it is the unique {monotone Riemannian metric} that induces a dually flat structure on the space of quantum states \cite{AmariNagaoka2007}.
Namely, the e-connection is torsion-free when we adopt this metric.
In this case, the e-representation of a tangent vector $X$ at $\rho$ is defined as $X^{(e)}=X(\log \rho)$, where $\log \rho$ is the matrix logarithm of $\rho$, and the e-geodesic from $\rho_1$ to $\rho_2$ is given by
\[
    \rho(t)=C(t)\exp((1-t)\log \rho_1+t \log \rho_2),
\]
where $\exp$ denotes the matrix exponential and $C(t)$ is the normalization constant for $\tr \rho(t)=1$ \cite{AmariNagaoka2007}.
Also, the canonical divergence between quantum states is given by the \emph{Umegaki quantum relative entropy}~\cite{Umegaki1962}:
\begin{align}\label{eq:umegaki}
    D_{\mathrm{U}}(\rho \mid\mid \sigma) = \tr[\rho\log(\rho)-\rho\log(\sigma)],
\end{align}
which can be viewed as a quantum generalization of the Kullback--Leibler divergence and has been found to play a central role in quantum information theory \cite{Holevo2013,Petz2008}.
Also, the generalized Pythagorean theorem holds as follows.

\begin{lemma}
For points $\rho,\sigma,\tau \in S(\C^n)$, let $\gamma_1$ be the e-geodesic from $\rho$ to $\sigma$ and $\gamma_2$ be the m-geodesic from $\sigma$ to $\tau$.
If $\gamma_1$ and $\gamma_2$ are orthogonal with respect to the Bogoliubov--Kubo--Mori metric \eqref{eq:Bmetric} at $\sigma$, then $D_{\mathrm{U}}(\tau \mid\mid \rho) = D_{\mathrm{U}}(\tau \mid\mid \sigma) + D_{\mathrm{U}}(\sigma \mid\mid \rho)$.
\end{lemma}

Therefore, the e-projection of $\rho$ onto an {m-totally geodesic} connected submanifold $M$ is uniquely characterized as the minimizer of the quantum relative entropy as follows.

\begin{lemma}
Let $M$ be an {m-totally geodesic} connected submanifold of $S(\C^n)$ and $\rho \in S(\C^n)$.
Then, a point $\sigma \in M$ satisfies $D_{\mathrm{U}}(\sigma \mid\mid \rho) = \min_{\tau \in M} D_{\mathrm{U}}(\tau \mid\mid \rho)$ if and only if the e-geodesic from $\rho$ to $\sigma$ is orthogonal to $M$ at $\sigma$ with respect to the Bogoliubov--Kubo--Mori metric \eqref{eq:Bmetric}.
\end{lemma}

Thus, the alternating e-projections onto $\Pi_1$ and $\Pi_2$ with respect to the Bogoliubov--Kubo--Mori metric \eqref{eq:Bmetric} is given by the alternating Umegaki relative entropy minimization:
	\begin{align*}
		\rho_{2k+1} &= \argmin_{\rho \in \Pi_1} D_{\mathrm{U}}(\rho \mid\mid \rho_{2k}), \\ \rho_{2k+2} &= \argmin_{\rho \in \Pi_2} D_{\mathrm{U}}(\rho \mid\mid \rho_{2k+1}).
	\end{align*}

\subsection{Alternating e-projections with respect to congruence invariant metric}
The space of $n \times n$ Hermitian positive definite matrices $P_n=\{ S \in \C^{n \times n} \mid S \succ O \}$ is regarded as a Riemannian manifold by introducing the \emph{congruence invariant metric}~\cite{Hiai2009}:
\begin{align}\label{eq:Cmetric}
    g^{\mathrm{C}}(X, Y) = \tr[S^{-1} X S^{-1} Y],
\end{align}
which is equal to the Hessian metric of $-\log\det(S)$.
Note that this metric does not belong to the class of monotone metrics for quantum states.
On the submanifold of $n \times n$ real positive definite matrices $\{ \Sigma \in \R^{n \times n} \mid \Sigma \succ O \} \subset P_n$, this metric coincides with the Fisher information metric \cite{AmariNagaoka2007} for the family of mean-zero $n$-dimensional Gaussian distributions ${\rm N}(0,\Sigma)$.
This Riemannian structure (with the Levi-Civita connection) of $P_n$ has been recently utilized for operator scaling \cite{Allen-Zhu2018}.
See Section~\ref{subsec:cap} for its detail.
Here, we instead focus on the dually flat structure of $P_n$ \cite{Ohara1996}.
Namely, we define the m-connection to be the affine connection induced by the natural embedding of $P_n$ into $\R^{n^2}$ and define the e-connection to be its dual with respect to the {congruence invariant metric} \eqref{eq:Cmetric}.
Then, $P_n$ is shown to be dually flat \cite{Ohara1996}.
In this case, the e-representation of a tangent vector $X$ at $S$ is defined as $X^{(e)}=X(S^{-1})$ and the e-geodesic from $S_1$ to $S_2$ is given by
\[
    S(t)=((1-t)S_1^{-1}+tS_2^{-1})^{-1}.
\]
Also, the canonical divergence is given by the \emph{Burg divergence}~\cite{Dhillon2007}:
\begin{align*}
    D_{\mathrm{B}}(S \mid\mid T) = \tr (S T^{-1}) - \log \det (S T^{-1})-n,
\end{align*}
which can be viewed as a quantum generalization of the Itakura--Saito divergence \cite{Fevotte2009}.
The generalized Pythagorean theorem holds as follows.

\begin{lemma}
For points $S,T,U \in P_n$, let $\gamma_1$ be the e-geodesic from $S$ to $T$ and $\gamma_2$ be the m-geodesic from $T$ to $U$.
If $\gamma_1$ and $\gamma_2$ are orthogonal with respect to the {congruence invariant metric} \eqref{eq:Cmetric} at $T$, then $D_{\mathrm{B}}(U \mid\mid S) = D_{\mathrm{B}}(U \mid\mid T) + D_{\mathrm{B}}(T \mid\mid S)$.
\end{lemma}

Therefore, the e-projection of $S$ onto an {m-totally geodesic} connected submanifold $M$ is uniquely characterized as the minimizer of the Burg divergence as follows.

\begin{lemma}
Let $M$ be an {m-totally geodesic} connected submanifold of $P_n$ and $S \in P_n$.
Then, a point $T \in M$ satisfies $D_{\mathrm{B}}(T \mid\mid S) = \min_{U \in M} D_{\mathrm{B}}(U \mid\mid S)$ if and only if the e-geodesic from $S$ to $T$ is orthogonal to $M$ at $T$ with respect to the {congruence invariant metric} \eqref{eq:Cmetric}.
\end{lemma}

Thus, the alternating e-projections onto $\Pi_1$ and $\Pi_2$ with respect to the congruence invariant metric \eqref{eq:Cmetric} is given by the alternating minimizations of the Burg divergence:
	\begin{align*}
		\rho_{2k+1} &= \argmin_{\rho \in \Pi_1} D_{\mathrm{B}}(\rho \mid\mid \rho_{2k}), \\ \rho_{2k+2} &= \argmin_{\rho \in \Pi_2} D_{\mathrm{B}}(\rho \mid\mid \rho_{2k+1}).
	\end{align*}
By using the Lagrange multiplier method, each iteration is solved as follows.

\begin{lemma}
	$\rho_{2k+1} = (\rho_{2k}^{-1}-I_n \otimes A)^{-1}$, where $A$ is uniquely determined from the condition $\tr_1 \rho_{2k+1} = m^{-1} I_m$.
	Similarly, $\rho_{2k+2} = (\rho_{2k+1}^{-1}-B \otimes I_m)^{-1}$, where $B$ is uniquely determined from  the condition $\tr_2 \rho_{2k+2}=n^{-1} I_n$.
\end{lemma}


\subsection{Numerical comparison}\label{subsec:comp}
Thus far, we have three alternating e-projections algorithms corresponding to three Riemannian metrics: SLD metric (operator Sinkhorn), Bogoliubov--Kubo--Mori metric and {congruence invariant metric}.
Here, we compare the outputs of these algorithms numerically with $m=n=2$.

We generated a $4 \times 4$ density matrix $\rho_0$ by $\rho_0=P^{\top}P/\tr(P^{\top}P)$, where $P$ is a $4 \times 4$ matrix of standard Gaussian random variables.
Then, we applied three algorithms starting from $\rho_0$.
In the minimization of the Umegaki quantum relative entropy and Burg divergence, we used the MATLAB package CVXQUAD~\cite{Fawzi2018}.
For the input
\[
    \rho_0 = \begin{bmatrix}
   0.2703 + 0.0000i &  0.0136 + 0.1050i  & 0.0567 - 0.0746i &  0.1015 - 0.0678i \\
   0.0136 - 0.1050i &  0.3054 + 0.0000i  & 0.1197 - 0.0485i & -0.0880 - 0.0650i \\
   0.0567 + 0.0746i &  0.1197 + 0.0485i  &   0.2162 + 0.0000i &  0.0938 - 0.0463i \\
   0.1015 + 0.0678i & -0.0880 + 0.0650i  & 0.0938 + 0.0463i &  0.2081 + 0.0000i
    \end{bmatrix},
\]
the outputs (after 200 iterations) of the alternating e-projections algorithms with respect to the SLD metric, Bogoliubov--Kubo--Mori metric and {congruence invariant metric} are
\[
   \rho_* =  \begin{bmatrix}
   0.2386 + 0.0000i & -0.0545 + 0.0790i  & 0.0803 - 0.0070i &  0.1130 - 0.0568i \\
  -0.0545 - 0.0790i &  0.2614 + 0.0000i  & 0.1484 - 0.0409i & -0.0803 + 0.0070i \\
   0.0803 + 0.0070i &  0.1484 + 0.0409i  & 0.2614 - 0.0000i &  0.0545 - 0.0790i \\
   0.1130 + 0.0568i & -0.0803 - 0.0070i  & 0.0545 + 0.0790i &  0.2386 - 0.0000i
    \end{bmatrix},
\]
\[
   \rho_* =  \begin{bmatrix}
   0.2363 + 0.0000i & -0.0601 + 0.0694i  & 0.0700 - 0.0039i &  0.1215 - 0.0405i \\
  -0.0601 - 0.0694i &  0.2637 + 0.0000i  & 0.1535 - 0.0279i & -0.0700 + 0.0039i \\
   0.0700 + 0.0039i &  0.1535 + 0.0279i  & 0.2637 + 0.0000i &  0.0601 - 0.0694i \\
   0.1215 + 0.0405i & -0.0700 - 0.0039i  & 0.0601 + 0.0694i &  0.2363 + 0.0000i
    \end{bmatrix},
\]
\[
   \rho_* =  \begin{bmatrix}
   0.2154 + 0.0000i & -0.0861 + 0.0094i  & 0.0126 - 0.0150i &  0.1484 + 0.0278i \\
  -0.0861 - 0.0094i &  0.2846 + 0.0000i  & 0.2196 + 0.0537i & -0.0126 + 0.0150i \\
   0.0126 + 0.0150i &  0.2196 - 0.0537i  & 0.2853 + 0.0000i &  0.0918 - 0.0096i \\
   0.1484 - 0.0278i & -0.0126 - 0.0150i &  0.0918 + 0.0096i &  0.2147 + 0.0000i
    \end{bmatrix},
\]
respectively.
It is clear that the three procedures converge to distinct points.
It is an interesting future problem to investigate applications of these two alternating e-projection algorithms.

\section{Discussion}\label{sec:others}
In this section, we briefly review yet another Riemannian approach to operator scaling in previous study which does not fit into quantum information geometry.
We also discuss a divergence characterization of the operator Sinkhorn algorithm.
For the sake of simplicity, we assume $m = n$ in the following.

\subsection{Another Riemannian approach for operator scaling}\label{subsec:cap}
In \cite{Allen-Zhu2018}, the authors devised an efficient algorithm for operator scaling base on a Riemannian structure of the positive definite cone induced by the congruence invariant metric \eqref{eq:Cmetric} with the Levi-Civita connection.
They focused on the \emph{capacity}~\cite{Gurvits2004} of a completely positive map $\Phi$ defined as
\begin{align}\label{eq:capacity}
    \capa(\Phi) = \inf_{X \succ O} \frac{\det \Phi(X)}{\det X}.
\end{align}
It is an extension of the capacity of a nonnegative matrix $A$ in matrix scaling \cite{Linial1998}:
\begin{align}\label{eq:capacity_mat}
    \capa(A) = \inf_{x > 0} \frac{\prod_{i=1}^n (A x)_i}{\prod_{i=1}^n x_i}.
\end{align}
It is known that if one finds $X \succ O$ achieving the (approximate) infimum of \eqref{eq:capacity}, then one can recover the (approximate) solution of operator scaling~\cite{Garg2019}.
Unfortunately, the problem \eqref{eq:capacity} is nonconvex in the Euclidean geometry.
The main idea of \cite{Allen-Zhu2018} is to consider an equivalent problem
\begin{align}\label{eq:log-capacity}
    \inf_{X \succ O} \log\det \Phi(X)- \log\det X,
\end{align}
and show that this problem is \emph{geodesically convex}.
Specifically, they regarded the positive definite cone $P_n$ as a Riemannian manifold by introducing the congruence invariant metric \eqref{eq:Cmetric} and the corresponding Levi-Civita connection.
It is known \cite{HiaiPetz2014} that the geodesic from $X_1$ to $X_2$ on this Riemannian manifold is given by
\begin{align}\label{eq:geodesic}
    X(t) = X_1^{1/2} (X_1^{-1/2} X_2 X_1^{-1/2})^t X_1^{1/2},
\end{align}
which is different from the e-geodesic \eqref{eq:SLD-geodesic} with respect to the SLD metric.
Then, in the problem \eqref{eq:log-capacity}, the objective function is convex along with geodesic \eqref{eq:geodesic} and this property leads to an efficient algorithm \cite{Allen-Zhu2018}.
It may be interesting to investigate the relation of this approach with our quantum information geometric results and we leave it for future work.

\subsection{Divergence characterization of operator Sinkhorn algorithm}
The Sinkhorn algorithm for matrix scaling is equivalent to alternating minimization of the Kullback--Leibler divergence  (see Corollary~\ref{cor:KL-Sinkhorn}). 
Thus, it is natural to ask if some quantum analogue of the Kullback--Leibler divergence yields the same characterization of the operator Sinkhorn algorithm.
This problem was already mentioned by Gurvits~\cite[Remark 4.8]{Gurvits2004} and even said to be ``a major open question" of operator scaling in the survey by Idel~\cite{Idel2016}.
Here, we discuss this problem from the viewpoint of quantum information geometry.

Although quantum generalization of the Kullback--Leibler divergence is not unique, the most important one may be the {Umegaki quantum relative entropy} \eqref{eq:umegaki}.
However, alternating minimization of the {Umegaki quantum relative entropy}, which is equivalent to alternating e-projections with respect to the {Bogoliubov--Kubo--Mori metric} \eqref{eq:Bmetric} discussed in Section~\ref{subsec:kubo}, does not coincide with the operator Sinkhorn algorithm as shown in Section~\ref{subsec:comp}.
As another possibility, we can think of the Belavkin--Staszewski relative entropy \cite{belavkin1982c}:
\begin{align}\label{eq:BS}
D_{\mathrm{BS}}(\rho \mid\mid \sigma)=-\tr[\rho \log(\rho^{-1/2} \sigma \rho^{-1/2})],
\end{align}
which is also viewed as a quantum generalization of the Kullback--Leibler divergence.
Since it seems difficult to explicitly solve alternating minimization of the Belavkin--Staszewski relative entropy, here we instead focus on its derivative at the output $\rho_*$ of the operator Sinkhorn algorithm.
Specifically, let
\[
\Delta(h,A) = \frac{D_{\mathrm{BS}}(\rho_*+hA \mid\mid \rho_0)-D_{\mathrm{BS}}(\rho_*-hA \mid\mid \rho_0)}{2 h}
\]
be the central difference quotient of the Belavkin--Staszewski relative entropy at $\rho_*$.
Figure~\ref{fig:diff_BS} plots $\Delta(h,A)$ with respect to $h$ in the same setting with Section~\ref{subsec:comp}, where the direction of the perturbation was set to
\[
   A = \begin{bmatrix}
   1 &  0  & 0 &  0 \\
   0 &  -1  & 0 & 0 \\
   0 &  0  &  -1 & 0 \\
   0 & 0 & 0 & 1
    \end{bmatrix},
\]
so that $\rho_* \pm hA$ remains in $\Pi_1 \cap \Pi_2$.
It indicates that $\Delta(h,A)$ does not converge to zero as $h \to 0$, which implies that the point $\rho_*$ is not a stationary point of the Belavkin--Staszewski relative entropy.
Therefore, the operator Sinkhorn algorithm does not minimize the Belavkin--Staszewski relative entropy.
For comparison, Figure~\ref{fig:diff_KL} gives the same plot for the case where $\rho_0$ is diagonal, which corresponds to matrix scaling and thus the Belavkin--Staszewski relative entropy reduces to the Kullback--Leibler divergence.
Clearly, $\Delta(h,A)$ converges to zero in this case.

\begin{figure}
\centering
\begin{subfigure}[b]{.45\linewidth}
\centering
\begin{tikzpicture}
\begin{axis}[
ymax=1.1, ymin=-0.1,
xlabel={log10 $h$},
ylabel={$\Delta(h,A)$},
grid=major,
width=\linewidth
]
\addplot[very thick, color=black,
filter discard warning=false, unbounded coords=discard] table [x index=0, y index=1, col sep=comma] {experiments/diff_BS.csv};
\end{axis}
\end{tikzpicture} 
	\caption{non-diagonal $\rho_0$}
	\label{fig:diff_BS}
\end{subfigure}
\begin{subfigure}[b]{.45\linewidth}
\centering
\begin{tikzpicture}
\begin{axis}[
xlabel={log10 $h$},
ylabel={$\Delta(h,A)$},
grid=major,
width=\linewidth
]
\addplot[very thick, color=black,
filter discard warning=false, unbounded coords=discard
] table {
  -12.0412    0.0000
  -11.7402    0.0001
  -11.4391         0
  -11.1381   -0.0000
  -10.8371   -0.0000
  -10.5360    0.0000
  -10.2350   -0.0000
   -9.9340    0.0000
   -9.6330    0.0000
   -9.3319   -0.0000
   -9.0309   -0.0000
   -8.7299   -0.0000
   -8.4288    0.0000
   -8.1278   -0.0000
   -7.8268   -0.0000
   -7.5257    0.0000
   -7.2247    0.0000
   -6.9237    0.0000
   -6.6227    0.0000
   -6.3216    0.0000
   -6.0206   -0.0000
   -5.7196   -0.0000
   -5.4185    0.0000
   -5.1175    0.0000
   -4.8165    0.0000
   -4.5154    0.0000
   -4.2144    0.0000
   -3.9134    0.0000
   -3.6124    0.0000
   -3.3113    0.0000
   -3.0103    0.0000
   -2.7093    0.0000
   -2.4082    0.0000
   -2.1072    0.0001
   -1.8062    0.0005
   -1.5051    0.0019
};
\end{axis}
\end{tikzpicture} 
	\caption{diagonal $\rho_0$}
	\label{fig:diff_KL}
\end{subfigure}
\caption{Central difference quotient of Belavkin--Staszewski relative entropy}
\end{figure}
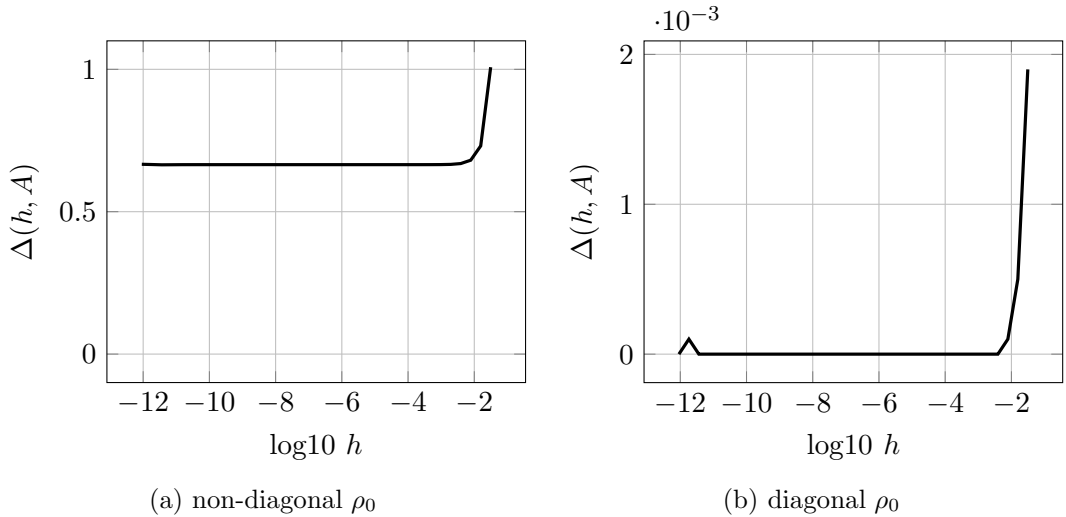

We also tried the sandwiched R\'enyi relative entropy of order $1/2$: 
\begin{align}\label{eq:Renyi}
D_{\mathrm{R}}(\rho \mid\mid \sigma)=-4 \log [\tr (\sigma^{1/2} \rho \sigma^{1/2})^{1/2}],
\end{align}
which is known to be locally equivalent to the SLD metric \cite{hayashi2002two,katariya2020geometric}, the measured relative entropy \cite{Petz2008,berta2017variational}:
\begin{align}\label{eq:measured}
D^{\mathbb{M}}(\rho \mid\mid \sigma) = \sup_{(\chi,M)} D_{\mathrm{KL}}(P_{\rho,M} \mid\mid P_{\sigma,M}),
\end{align}
where the supremum is taken over all finite sets $\chi$ and positive operator valued measures (POVMs) $M$ on $\chi$ and $P_{\rho,M}$ is a probability measure on $\chi$ defined by $P_{\rho,M}(x) = \tr[M(x)\rho]$, and another quantum generalization of the Kullback--Leibler divergence considered by Nagaoka \cite{Nagaoka1994}:
\begin{align}\label{eq:riccati}
D_{\mathrm{N}}(\rho \mid\mid \sigma)= 2 \tr [\rho \log (\rho \# \sigma^{-1})].
\end{align}
Figures~\ref{fig:diff_Renyi} and \ref{fig:diff_measured} show the results.
Again, these figures indicate that the operator Sinkhorn algorithm does not minimize these divergences.

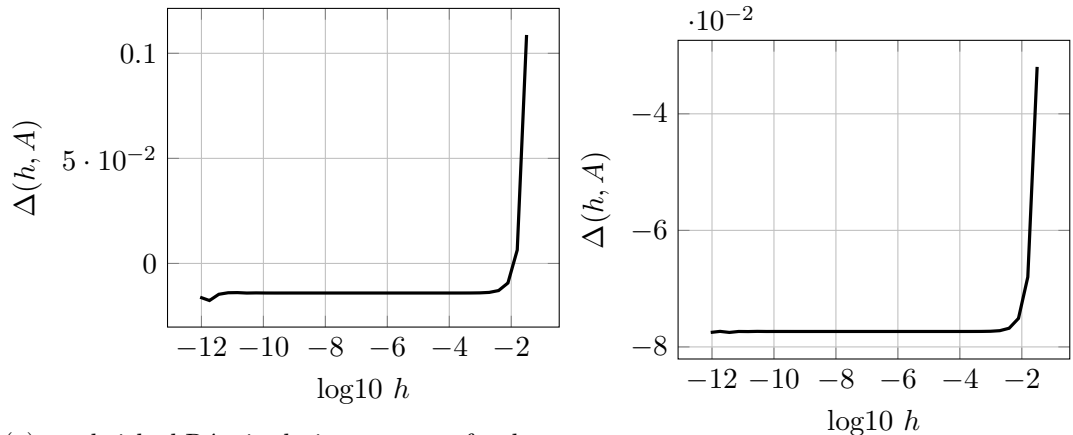
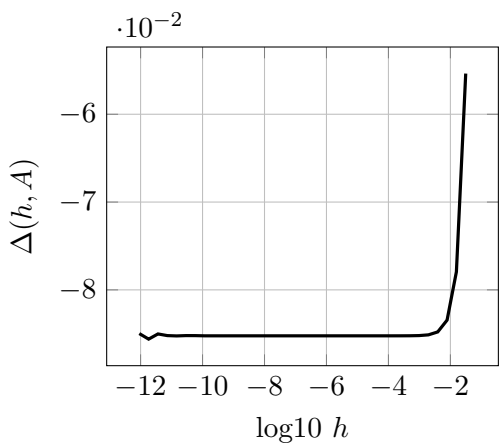
\begin{figure}
\centering
\begin{subfigure}[b]{.45\linewidth}
\centering
\begin{tikzpicture}
\begin{axis}[
xlabel={log10 $h$},
ylabel={$\Delta(h,A)$},
grid=major,
width=.95\linewidth
]
\addplot[very thick, color=black,
filter discard warning=false, unbounded coords=discard] table [x index=0, y index=1, col sep=comma] {experiments/diff_Renyi.csv};
\end{axis}
\end{tikzpicture} 
	\caption{sandwiched R\'enyi relative entropy of order 1/2}
	\label{fig:diff_Renyi}
\end{subfigure}
\begin{subfigure}[b]{.45\linewidth}
\centering
\begin{tikzpicture}
\begin{axis}[
xlabel={log10 $h$},
ylabel={$\Delta(h,A)$},
grid=major,
width=.95\linewidth
]
\addplot[very thick, color=black,
filter discard warning=false, unbounded coords=discard] table [x index=0, y index=1, col sep=comma] {experiments/diff_measured.csv};
\end{axis}
\end{tikzpicture} 
	\caption{measured relative entropy}
	\label{fig:diff_measured}
\end{subfigure}
    %
    \vspace{1em}
    \par\bigskip 
\begin{subfigure}[b]{.45\linewidth}
\centering
\begin{tikzpicture}
\begin{axis}[
xlabel={log10 $h$},
ylabel={$\Delta(h,A)$},
grid=major,
width=.95\linewidth
]
\addplot[very thick, color=black,
filter discard warning=false, unbounded coords=discard] table [x index=0, y index=1, col sep=comma] {experiments/diff_riccati.csv};
\end{axis}
\end{tikzpicture} 
	\caption{Nagaoka divergence}
	\label{fig:diff_riccati}
\end{subfigure}
\caption{Central difference quotient of various divergences}
\end{figure}

\subsection{Divergence characterization of capacity}
In matrix scaling, it is known that the capacity \eqref{eq:capacity_mat} is characterized as the minimum of the Kullback--Leibler divergence \cite{Linial1998}:
\[
    -\log\capa(A) = \inf_{B: \text{doubly stochastic}} D_{\mathrm{KL}}(B \mid\mid A) = D_{\mathrm{KL}}(A_* \mid\mid A),
\]
where $A_*$ is the output of the Sinkhorn algorithm from the input $A$.
Hence, one may still hope that a similar characterization holds between the capacity \eqref{eq:capacity} and the Umegaki quantum relative entropy, or other divergences cosidered above, in operator scaling. 
However, we will numerically show that this is not the case in the following.

In this experiment, we compare the capacity with the various divergences for random initial Choi representations $\rho_0$.
We compute the capacity by the operator Sinkhorn algorithm. 
It is shown in~\cite{Gurvits2004} that $\capa(\Phi)=1$ for a doubly stochastic map $\Phi$ and the capacity of $\Phi^{(2k)}$ is given as
\begin{align*}
    \capa(\Phi^{(2k)}) = \capa(\Phi^{(0)})\cdot \prod_{i=1}^k \det(L^{(2i)})^{1/n}\det(R^{(2i+1)})^{1/n},
\end{align*}
where $L^{(2i)}$ and $R^{(2i+1)}$ are scaling matrices used in the operator Sinkhorn's iterations for $\Phi^{(2i)}$ and $\Phi^{(2i+1)}$, respectively.
Hence, once the operator Sinkhorn algorithm converges to a (approximately) doubly stochastic $\Phi^{(2k)}$, we can compute $\capa(\Phi^{(0)})$.
We run the operator Sinkhorn algorithm until the following stopping criterion is satisfied:
\begin{align}\label{eq:stop-criteria}
    \norm*{\tr_1 \rho - \frac{1}{m}I_m}_{\mathrm{F}}^2 + \norm*{\tr_2 \rho - \frac{1}{n}I_n}_{\mathrm{F}}^2 < 10^{-8}.
\end{align}

Figure~\ref{fig:logcap} plots the Umegaki quantum relative entropy $D_{\mathrm{U}}(\rho_* \mid\mid \rho_0)$ and the negative logarithm of capacity $-\log\capa(\rho_0)$ over 30 random initial $\rho_0$. 
It is clearly shown that they consistently disagree. 
Figures~\ref{fig:logcap_BS}-\ref{fig:logcap_riccati} give the same plots for the Belavkin--Staszewski relative entropy \eqref{eq:BS}, sandwiched R\'enyi relative entropy of order 1/2 \eqref{eq:Renyi}, measured relative entropy \eqref{eq:measured} and Nagaoka divergence \eqref{eq:riccati}, respectively.
The Nagaoka divergence gave the smallest numerical difference (Figure~\ref{fig:logcap_riccati}), but there are still non-negligible error of order $10^{-2}$.
Hence, these divergences do not satisfy the desired identity.

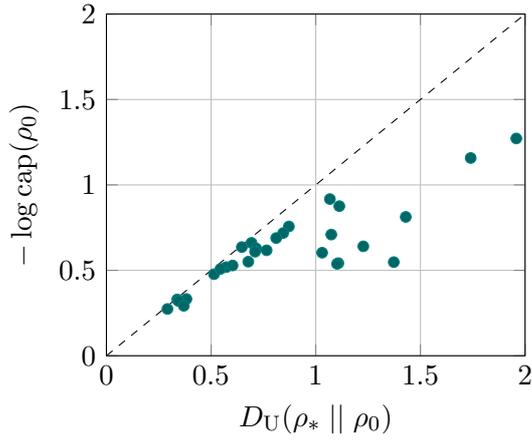
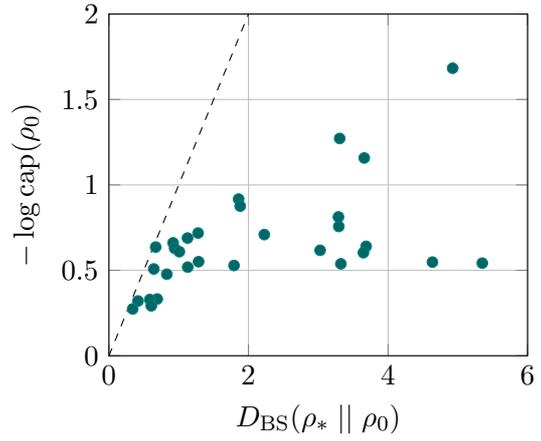
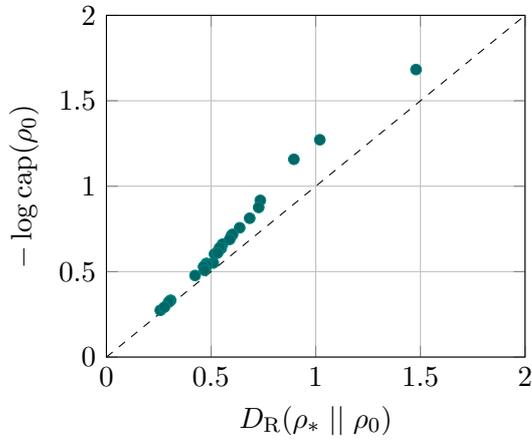
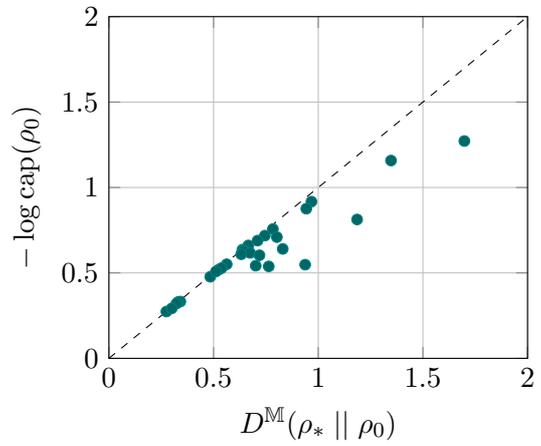
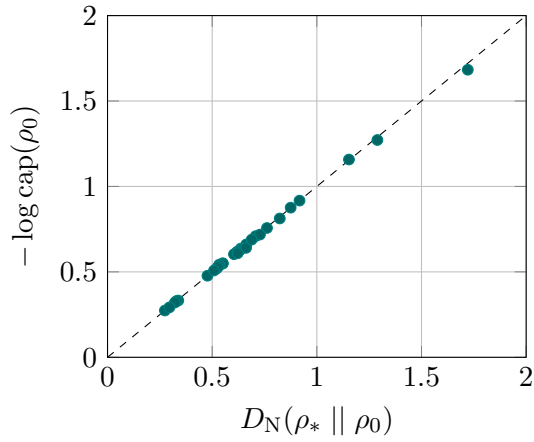
\begin{figure}
    \centering
    \begin{subfigure}[b]{.45\linewidth}
    \centering
    \begin{tikzpicture}
        \begin{axis}[ 
        enlargelimits=false,
        xlabel={$D_{\mathrm{U}}(\rho_* \mid\mid \rho_0)$}, xmin=0, xmax=2,
        ylabel={$-\log\capa(\rho_0)$}, ymin=0, ymax=2,
        grid=major,
        width=\linewidth]
            \addplot+[only marks] table [x index=0, y index=1, col sep=comma] {experiments/logcap.csv};
            \addplot[dashed, domain=0:2] {x};
        \end{axis}
    \end{tikzpicture}
    \subcaption{Umegaki quantum relative entropy}
    \label{fig:logcap}
    \end{subfigure}
    \begin{subfigure}[b]{.45\linewidth}
    \centering
    \begin{tikzpicture}
        \begin{axis}[ 
        enlargelimits=false,
        xlabel={$D_{\mathrm{BS}}(\rho_* \mid\mid \rho_0)$}, xmin=0, xmax=6,
        ylabel={$-\log\capa(\rho_0)$}, ymin=0, ymax=2,
        grid=major,
        width=\linewidth]
            \addplot+[only marks] table [x index=0, y index=1, col sep=comma] {experiments/logcap_BS.csv};
            \addplot[dashed, domain=0:2] {x};
        \end{axis}
    \end{tikzpicture}
    \subcaption{Belavkin--Staszewski relative entropy}
    \label{fig:logcap_BS}
    \end{subfigure}
    \vspace{1em}
    \par\bigskip 
    \begin{subfigure}[b]{.45\linewidth}
    \centering
    \begin{tikzpicture}
        \begin{axis}[ 
        enlargelimits=false,
        xlabel={$D_{\mathrm{R}}(\rho_* \mid\mid \rho_0)$}, xmin=0, xmax=2,
        ylabel={$-\log\capa(\rho_0)$}, ymin=0, ymax=2,
        grid=major,
        width=\linewidth]
            \addplot+[only marks] table [x index=0, y index=1, col sep=comma] {experiments/logcap_Renyi.csv};
            \addplot[dashed, domain=0:2] {x};
        \end{axis}
    \end{tikzpicture}
    \subcaption{sandwiched R\'enyi relative entropy of order 1/2}
    \label{fig:logcap_Renyi}
    \end{subfigure}
    \begin{subfigure}[b]{.45\linewidth}
    \centering
    \begin{tikzpicture}
        \begin{axis}[ 
        enlargelimits=false,
        xlabel={$D^{\mathbb{M}}(\rho_* \mid\mid \rho_0)$}, xmin=0, xmax=2, 
        ylabel={$-\log\capa(\rho_0)$}, ymin=0, ymax=2,
        grid=major,
        width=\linewidth]
            \addplot+[only marks] table [x index=0, y index=1, col sep=comma] {experiments/logcap_measured.csv};
            \addplot[dashed, domain=0:2] {x};
        \end{axis}
    \end{tikzpicture}
    \subcaption{measured relative entropy}
    \label{fig:logcap_measured}
    \end{subfigure}
    \vspace{1em}
    \par\bigskip 
    \begin{subfigure}[b]{.45\linewidth}
    \centering
    \begin{tikzpicture}
        \begin{axis}[ 
        enlargelimits=false,
        xlabel={$D_{\mathrm{N}}(\rho_* \mid\mid \rho_0)$}, xmin=0, xmax=2,
        ylabel={$-\log\capa(\rho_0)$}, ymin=0, ymax=2,
        grid=major,
        width=\linewidth]
            \addplot+[only marks] table [x index=0, y index=1, col sep=comma] {experiments/logcap_riccati.csv};
            \addplot[dashed, domain=0:2] {x};
        \end{axis}
    \end{tikzpicture}
    \subcaption{Nagaoka divergence}
    \label{fig:logcap_riccati}
    \end{subfigure}
    \caption{Comparison of various divergences and negative logarithm of capacity}
\end{figure}

%% file: concl.tex
\section{Conclusion}\label{sec:conclusion}
In this study, we investigated the operator Sinkhorn algorithm \cite{Gurvits2004} from the viewpoint of quantum information geometry \cite{AmariNagaoka2007}.
Our main finding is that the operator Sinkhorn algorithm coincides with alternating e-projections with respect to the symmetric logarithmic derivative (SLD) metric, which is a Riemannian metric on the quantum state space and plays an important role in quantum estimation theory \cite{Helstrom1967}. 
This result is viewed as a generalization of the result by Csisz{\'a}r~\cite{Csiszar1975} to operator scaling. 

Whereas the matrix Sinkhorn algorithm is viewed as alternating minimization of the Kullback--Leibler divergence, such divergence characterization is still unclear for the operator Sinkhorn algorithm.
The main obstacle is that the e-connection induced by the SLD metric is not torsion-free and thus existing tools for dually flat spaces such as the generalized Pythagorean theorem are not directly applicable.
It is an interesting future work to develop theories of statistical manifolds admitting torsion \cite{Matsuzoe2010,Henmi2019} and obtain divergence characterization of the operator Sinkhorn algorithm.

%% file: ack.tex
\section*{Acknowledgment}
We thank Mark Wilde for helpful comments.
This work was supported by JSPS KAKENHI Grant Numbers 16H06533, 19K20212 and 19K20220.